\documentclass[12pt,a4paper]{amsart}
\usepackage[utf8]{inputenc}
\usepackage{amsmath,amsthm,amssymb,amsfonts,amscd}
\usepackage{graphicx}
\usepackage[all]{xy}
\usepackage{indentfirst}
\usepackage{txfonts}

\theoremstyle{plain}
\newtheorem{thm}{Theorem}[section]

\newtheorem{prop}[thm]{Proposition}

\theoremstyle{definition}
\newtheorem{dfn}[thm]{Definition}

\theoremstyle{remark}
\newtheorem{rem}[thm]{Remark}

\renewcommand{\tilde}{\widetilde}

\makeatletter
 
 \@addtoreset{equation}{section}
\makeatother

\newcommand{\R}{{\mathbb R}}
\newcommand{\Z}{{\mathbb Z}}

\title{Desingularizing special generic maps into 3-dimensional space}
\author{Masayuki Nishioka}
\address{Graduate~School~of~Mathematics, Kyushu University, Motooka~744}
\date{}

\begin{document}

\maketitle

\begin{abstract}
A smooth map between smooth manifolds is called a special generic map
if it has only definite fold points as its singularities.
In this paper, we study the desingularization problem of special generic maps of
closed orientable $n$-dimensional manifolds $M$ into $\R^3$ for $n \geq 5$.
We say that a smooth map $f:M \to \R^p$ is lifted to an immersion or an embedding
$F:M \to \R^k$ $(k>p)$ if $f$ is factorized as $f= \pi \circ F$ for
the standard projection $\pi:\R^k \to \R^p$.
In this paper, we first prove that if $n=5$ or $6$ and $M$ is simply connected, then
a special generic map $f:M \to \R^3$ can be lifted to an embedding into $\R^{n+1}$
if and only if the normal bundle $\nu_{f}$
of the singular point set of $f$ in $M$ is trivial as a vector bundle.
Second, we prove that
for a special generic map $f:M \to \R^3$
of a closed orientable $n$-dimensional manifold $M$,
if $n \geq 5$, $k \geq (3n+3)/2$ and $\nu_{f}$ is trivial, then
$f$ can be lifted to an embedding into $\R^k$.
\end{abstract}

\section{Introduction}

Haefliger \cite{Haefliger} proved that
for a generic smooth map $f:M \to \R^2$ of a closed surface,
there exists an immersion $F : M \to \R^3$ such that $f=\pi \circ F$
for the standard projection $\pi:\R^3 \to \R^2$ if and only if
the number of cusps on each component $C$ of the singular point set of $f$ is even or odd
according as the tubular neighborhood of $C$ in $M$ is orientable or non-orientable.
Here, a smooth map $f:M \to \R^2$ is \emph{generic}
if it has only fold points and cusp points as its singularities.
In particular, not every generic smooth map can be so lifted.

Based on Haefliger's result mentioned above, let us consider the following problem:
``Given a smooth map $f:M \to \R^p$ of a closed $n$-dimensional manifold and
an integer $k$ with $k>n \geq p$,
determine whether or not $f$ can be factorized as $f=\pi \circ F$ for
an immersion or embedding $F:M \to \R^k$ 
and for the standard projection $\pi:\R^k \to \R^p$.''
Such a non-singular map $F$ is called a \emph{lift} of $f$.
We can consider $F$ as a desingularization of $f$.
This lifting problem has been studied in various situations.

Yamamoto \cite{Yamamoto} proved that
a generic smooth map of a closed surface into $\R^2$
can always be lifted to an embedding into $\R^4$.
Saito \cite{Saito} proved that a special generic map
$f:M \to \R^n$ of a closed orientable $n$-dimensional manifold can always be
lifted to an immersion into $\R^{n+1}$.
Here, special generic maps are smooth maps with only definite fold points as their singularities.
Blank and Curley \cite{BC} studied the condition for
a generic smooth map $f:M \to N$ between smooth manifolds of the same dimension
to be lifted to an immersion into a line bundle $\pi:E \to N$.
Note that these results concern the desingularization of generic maps between manifolds of the same dimension.

Let us recall the definition of a special generic map. 
A smooth map $f:M \to \R^p$ of a closed $n$-dimensional manifold with $n \geq p$
is called a \emph{special generic map}
if it has only definite fold points as its singularities (for details, see Section 2).
Note that a special generic map into the line is nothing but
a Morse function with only critical points of minimum or maximum indices;
in particular, the source manifold of such a map is homeomorphic to the sphere if it is connected.

Special generic maps were first defined by Burlet and de Rham \cite{burlet},
who showed that a closed $3$-dimensional manifold $M$
admits a special generic map into the plane if and only if
$M$ is diffeomorphic to the $3$-sphere or to the connected sum of a finite number of
total spaces of $S^2$-bundles over $S^1$.
Porto and Furuya \cite{porto} studied the condition for
a closed $n$-dimensional manifold $M$ to admit a special generic map into the plane.
Saeki \cite{saeki} proved that a closed $n$-dimensional manifold $M$ with $n \geq 3$
admits a special generic map into the plane if and only if
$M$ is diffeomorphic to the $n$-dimensional homotopy sphere
($n$-dimensional standard sphere for $n \leq 6$) or to
the connected sum of a finite number of
total spaces of homotopy $(n-1)$-sphere bundles
($(n-1)$-sphere bundles for $n \leq 6$) over $S^1$
and a homotopy $n$-sphere (for $n \geq 7$).

\`{E}lia\v{s}berg \cite{eliasberg} proved that
a closed orientable $n$-dimensional manifold
admits a special generic map into $\R^n$ if and only if
$M$ is stably parallelizable, that is, the Whitney sum of
the tangent bundle of $M$ and the trivial line bundle over $M$ is
trivial as a vector bundle.

Let us now return to the lifting problem.
Let us first review some results about the lifting problem for smooth functions.
Burlet and Haab \cite{BH} proved that
a Morse function $f:M \to \R$ of a closed surface
can always be lifted to an immersion into $\R^3$.
Saeki and Takase \cite{ST} proved that
a special generic map $f:M \to \R$ of
a closed orientable $n$-dimensional manifold with $n \geq 1$
can always be lifted to an immersion into $\R^{n+1}$.
They also proved that a special generic map $f:M \to \R$ of
a closed connected $n$-dimensional manifold with $n \geq 2$
can be lifted to an embedding into $\R^{n+1}$ if and only if
$M$ is diffeomorphic to the $n$-dimensional sphere with the standard smooth structure.
Yamamoto \cite{Yamamoto2} gave a necessary and sufficient condition for a Morse function
on the circle to be lifted to an embedding into $\R^2$.

Let us now review some results about the lifting problem for smooth maps into the plane.
Kushner, Levine and Porto \cite{KLP} studied the lifting problem
for generic smooth maps of $3$-dimensional manifolds into $\R^2$.
Levine \cite{Levine} gave a necessary and sufficient condition
for a generic smooth map $f:M \to \R^2$ of a closed orientable $3$-dimensional manifold
to be lifted to an immersion into $\R^4$.
Saeki and Takase \cite{ST} proved that
a special generic map $f:M \to \R^2$ of
a closed orientable $n$-dimensional manifold with $n \geq 2$
can always be lifted to an immersion into $\R^{n+1}$.
They also proved that
a special generic map $f:M \to \R^2$ of
a closed non-orientable $n$-dimensional manifold with $n \geq 2$
can be lifted to an immersion into $\R^{n+1}$ if and only if
$n=2$, $4$ or $8$, and the tubular neighborhood of
the singular point set in $M$ is orientable.
On the other hand, they also proved that
a special generic map $f:M \to \R^2$ of
a closed connected $n$-dimensional manifold with $n \geq 3$
can be lifted to an embedding into $\R^{n+1}$ if and only if
$M$ is diffeomorphic either to $S^n$
or to the connected sum of a finite number of copies of $S^1 \times S^{n-1}$.
Note that this result does not hold for $n=2$.
Actually, in \cite{ST}, it is proved that
there exists a special generic map $f:S^2 \to \R^2$ which
cannot be lifted to any embedding $F:S^2 \to \R^3$
(but which can be lifted to an immersion $F:S^2 \to \R^3$, since $S^2$ is orientable).

When $n-p=1$, Saeki and Takase \cite{ST} proved that
a special generic map $f:M \to \R^p$ of
a closed orientable $n$-dimensional manifold
can be lifted to an immersion into $\R^{n+1}$ if and only if
the homology class $[S(f)] \in H_{p-1}(M;\Z)$
represented by $S(f)$ vanishes.
Here, $S(f)$ is the set of all singular points of $f$ in $M$.
Note that for a special generic map $f:M \to \R^p$,
the $\Z_2$-homology class $[S(f)]_2 \in H_{p-1}(M;\Z_2)$ is
Poincar\'{e} dual to the Stiefel-Whitney class $w_{n-p+1}(M)$
(see \cite{thom}).
Therefore, if $f$ can be lifted to an immersion into $\R^{n+1}$, then
$M$ is spin, that is, the second Stiefel-Whitney class
$w_{2}(M) \in H^2(M;\mathbb{Z}_2)$ of $M$ vanishes.

When $(n,p)=(5,3)$, $(6,3)$, $(6,4)$ or $(7,4)$,
Saeki and Takase \cite{ST} proved that
a special generic map $f:M \to \R^p$ of
a closed orientable $n$-dimensional manifold
can be lifted to an immersion into $\R^{n+1}$ if and only if $M$ is spin.
They also showed that one can take an embedding as a lift if $(n,p)=(6,3)$.

In this paper, we study the lifting problem for special generic maps of
closed $n$-dimensional manifolds into $\R^3$ for $n \geq 5$.
First, we prove that a special generic map $f:M \to \R^3$ of
a closed simply connected $n$-dimensional manifold $M$, $n=5$ or $6$,
can be lifted to an embedding into $\R^{n+1}$ if and only if
the singular point set of $f$ has a trivial normal bundle in $M$.
Second, we show that for a special generic map $f:M \to \R^3$ of a closed orientable
$n$-dimensional manifold with $n \geq 5$,
the map $f$ can be lifted an embedding into $\R^k$,
if $k \geq (3n+3)/2$ and the normal bundle of
the singular point set $S(f)$ of $f$ in $M$ is trivial.

The paper is organized as follows.
In Section~2, we review various topological properties of special generic maps.
In Section~3, we give a necessary condition for a special generic map
to be lifted to a codimension two immersion
in terms of the normal bundle of the singular point set.
In Section~4,
we construct an embedding lift $F:M \to \R^{n+1}$
for a special generic map $f:M \to \R^3$ of
a closed simply connected $n$-dimensional manifold ($n=5,6$)
with trivial normal bundle of $S(f)$ in $M$,
by using the fact that $O(n-2)$ is a deformation retract of $\mathrm{Diff}(S^{n-3})$,
where $\mathrm{Diff}(S^{n-3})$ is the space of all self-diffeomorphisms of $S^{n-3}$.
This last fact has been proved by Smale \cite{smale} $(n=5)$ 
and Hatcher \cite{Hatcher} $(n=6)$.
By using a similar method, we construct an embedding lift $F:M \to \R^k$
for a special generic map $f:M \to \R^3$ of
a closed orientable $n$-dimensional manifold ($n \geq 5$ and $k \geq (3n+3)/2$)
with trivial normal bundle of $S(f)$ in $M$,
by using the fact that $\mathrm{Emb}(S^{n-3},\R^{k-3})$ is $2$-connected,
where $\mathrm{Emb}(S^{n-3},\R^{k-3})$ is the space of all embeddings of $S^{n-3}$ into $\R^{k-3}$.
This last fact has been proved by Budney \cite{budney}.

Throughout this paper, all manifolds and maps are of class $C^{\infty}$,
unless otherwise indicated.
For groups $G_1$ and $G_2$, ``$G_1 \cong G_2$'' means that they are isomorphic;
for smooth manifolds $M_1$ and $M_2$, ``$M_1 \cong M_2$'' means that they are diffeomorphic;
and for vector bundles $E_1$ and $E_2$, ``$E_1 \cong E_2$'' means that they are isomorphic.
The symbol $\mathbb{R}^n$ denotes the $n$-dimensional Euclidean space;
$D^n$ denotes the closed unit disk in $\mathbb{R}^n$; and
$S^n$ denotes the $n$-dimensional unit sphere in $\mathbb{R}^{n+1}$.
For a manifold $M$ with boundary,
$\mathrm{Int}\, M$ and $\partial M$ denote
the interior and the boundary of $M$, respectively.

\section{Preliminaries}

In this section, we review several results about
topological properties of special generic maps and their Stein factorizations,
which will be necessary in the proof of our main theorems.

\subsection{Special generic maps}

Let $f:M \to \R^p$ be a smooth map of a closed $n$-dimensional manifold, $n \geq p$.
A point $q \in M$ is called a \emph{singular point} of $f$ if
the rank of the differential $df_{q}:T_{q}M \to T_{f(q)}\R^p$
is strictly less than $p$.
We denote by $S(f)$ the set of all singular points of $f$
and call it the \emph{singular point set} of $f$.
A point $q \in M$ is called a \emph{definite fold point} if
there exist local coordinates $x=(x_1,x_2,\ldots,x_n)$ around $q$ and
$y=(y_1,y_2,\ldots,y_p)$ around $f(q)$ such that
\[
\left\{
\begin{array}{l}
y_i \circ f = x_i,\ 1 \leq i \leq p-1, \\
y_p \circ f = x_{p}^2+x_{p+1}^2+\cdots+x_{n}^2.
\end{array}
\right.
\]
When $f$ has no singular points except for definite fold points,
$f$ is called a \emph{special generic map}. 

When $p=1$, special generic maps are nothing but Morse functions
with only critical points of indices $0$ or $n$.

Note that for a special generic map $f:M \to \R^p$,
the singular point set $S(f)$ is a closed $(p-1)$-dimensional submanifold of $M$ and
the restriction of $f$ to $S(f)$ is a codimension one immersion into $\R^p$.

\subsection{Stein factorization}

Let $f:X \to Y$ be a continuous map between topological spaces.
For two points $x_{1}$ and $x_{2}$ in $X$, we define $x_{1} \sim x_{2}$ if
$x_{1}$ and $x_{2}$ are in the same connected component of
the pre-image $f^{-1}(y)$ for a point $y$ in $Y$.
This relation ``$\sim$'' is an equivalence relation, and therefore, we can take
the quotient space $W_{f}$ and the quotient map $q_{f}:X \to W_{f}$
with respect to this relation.
Then it is not difficult to prove that there exists a unique continuous map
$\bar{f}:W_{f} \to Y$ such that the diagram
\[
\xymatrix{
X \ar[dr]^{f} \ar[d]_{q_{f}}& \\
W_{f} \ar[r]_{\bar{f}}& Y \\
}
\]
commutes.
The above diagram is called the \emph{Stein factorization} of $f$.

In general, the quotient space in the Stein factorization of a smooth map
is not always a topological manifold.
However, for a special generic map $f:M \to \R^p$ of a closed $n$-dimensional manifold, $n > p$,
we can give a structure of a smooth $p$-dimensional manifold with boundary
to $W_f$ so that $\bar{f}:W_f \to \R^p$ is an immersion and
$q_f:M \to W_f$ is a smooth map.

Note that for a special generic map $f:M \to \R^n$ of a closed $n$-dimensional manifold
into the Euclidean space of the same dimension,
we have $M=W_f$, since the pre-image $f^{-1}(y)$ is a finite set for any $y \in \R^n$.
So the Stein factorization does not give any new information.
The following result is very useful to study special generic maps
(see \cite{burlet,saeki}).

\begin{thm}
\label{stein}
Let $f:M \to \R^p$  be a special generic map of
a closed connected $n$-dimensional manifold $M$ into $\R^p$, $n>p$.
Then the following holds.
\begin{enumerate}
 \item The quotient space $W_{f}$ has the structure of a smooth $p$-dimensional manifold with non-empty boundary.
 \item The map $q_f:M \to W_{f}$ is a smooth map.
 \item The map $\bar{f}:W_{f} \to \R^p$ is a smooth immersion.
 \item The singular point set $S(f)$ is a closed $(p-1)$-dimensional submanifold of $M$, and
the restriction of $q_{f}$ to $S(f)$ is a diffeomorphism onto $\partial W_{f}$.
 \item The induced map $(q_{f})_{*}:\pi_{1}(M) \to \pi_{1}(W_{f})$ is a group isomorphism.
 \item We have $q_{f}(M \setminus S(f)) = \mathrm{Int}\,W_{f}$ and
       $q_{f}|_{M \setminus S(f)}:M \setminus S(f) \to \mathrm{Int}\,W_{f}$ is a smooth
$S^{n-p}$-bundle over $\mathrm{Int}\,W_{f}$.
\end{enumerate}
\end{thm}
 
The above theorem is essentially proved for $(n,p)=(3,2)$ in \cite{burlet}.
We can prove it for general $(n,p)$, $n>p$, by using a similar method.


We say that a $D^n$-bundle (or an $S^n$-bundle) is \emph{linear} if its structure group can be
reduced to the orthogonal group $\mathit{O}(n)$. 
Saeki \cite{saeki} proved the following theorem about the topology of the source manifolds
of  special generic maps.

\begin{thm}
\label{fiber}
Suppose a closed $n$-dimensional manifold $M$ admits a special generic map into $\R^p$ with $n > p$.
Then there exists a topological $D^{n-p+1}$-bundle $E$ over $W_{f}$ with $M$ being homeomorphic to $\partial E$.
Furthermore, if $n-p \leq 3$, then we can arrange so that $E$ is a linear $D^{n-p+1}$-bundle over $W_f$
and that $M$ is diffeomorphic to $\partial E$.
\end{thm}

By using Theorem~\ref{fiber}, Saeki \cite{saeki,saeki2} classifies the diffeomorphism types of
the simply connected $n$-dimensional manifolds $(n=5,6)$ which admit special generic maps into $\R^3$.
That is, the following result holds.

\begin{thm}
\label{classification}
Let $M$ be a closed simply connected $n$-dimensional manifold with $n=5,6$.
Then $M$ admits a special generic map into $\R^3$ if and only if
$M$ is diffeomorphic to $S^n$ or  to the connected sum of $S^{n-2}$-bundles over $S^2$.
\end{thm}

\section{Normal bundle of the singular point set}

In this section, we give a necessary condition
for a special generic map $f:M \to \R^p$
of a closed orientable $n$-dimensional manifold $M$ with $n>p$
to be lifted to a codimension two immersion.

\begin{prop}\label{lemmaA}
Let $M$ be a closed orientable manifold of dimension $n$ and
let $F$ be an immersion of $M$ into $\R^{n+2}$ such that
$f = \pi \circ F$ is a special generic map, where
$\pi: \R^{n+2} \to \R^p$ is the standard projection, $n > p \geq 1$.
Then the normal bundle $\nu_f$ of $S(f)$ in $M$ is stably trivial.
Furthermore, if $n>2p-2$, then $\nu_f$ is trivial.
\end{prop}

Here, a vector  bundle $E$ over a topological space $B$ is said to be \emph{stably trivial} if
the Whitney sum of $E$ and a finite dimensional trivial vector bundle over $B$
is trivial as a vector bundle.

\begin{proof}[Proof of Proposition \ref{lemmaA}]

Let $i^{*}(TM)$ be the pullback of $TM$ induced by the inclusion map $i:S(f) \to M$ and
let $\tilde{i}:i^{*}(TM) \to TM$ be the natural map over $i$.
Then $\nu_f$ is identified with $\mathrm{ker}(df \circ \tilde{i})$, which
is an $(n-p+1)$-plane subbundle of $i^{*}(TM)$.

On the other hand, since $\bar{f}:W_f \to \R^p$ is an immersion of a $p$-dimensional manifold
and $q_f|_{S(f)}:S(f) \to \partial W_f$ is a diffeomorphism, we have that $S(f)$ is orientable.
Therefore, since $M$ is orientable and $i^{*}(TM) \cong TS(f) \oplus \nu_f$,
we have that $\nu_f$ is orientable.

Let $G$ be the restriction of $dF \circ \tilde{i}$ to $\nu_f$.
Since $f=\pi \circ F$ and $F$ is an immersion,
$G$ is a fiberwise monomorphism of
the $(n-p+1)$-plane bundle $\nu_f$ into the $(n-p+2)$-plane bundle $\mathrm{ker}(d\pi)$,
which is trivial.
Note that $G$ is a bundle morphism over $F \circ i$.
Therefore, $\nu_f$ is a subbundle of $\zeta=(F \circ i)^{*}(\mathrm{ker}(d\pi))$.
By using an inner product on $\zeta$, we find a $1$-dimensional subbundle $\xi$ of $\zeta$
such that $\nu_f \oplus \xi \cong \zeta$.
Since $\nu_f$ is an orientable $(n-p+1)$-plane bundle and
$\zeta$ is the trivial $(n-p+2)$-plane bundle, we have that
$\xi$ is an orientable line bundle and hence is trivial.
This means that $\nu_f$ is stably trivial.

Furthermore, if $n>2p-2$, i.e. if $n-p+1 > p-1$, then 
the dimension of the base space of $\nu_f$ is strictly less than
the dimension of the fibers of $\nu_f$.
Therefore, the stable triviality of $\nu_f$ implies the triviality of $\nu_f$.
This completes the proof.
\end{proof}

\section{Lifting problems for special generic maps}

In this section, we consider the problem of lifting special generic maps into $\R^3$
to codimension one embeddings.
Recall that if $n=5$ or $6$,
a closed simply connected $n$-dimensional manifold $M$
admits a special generic map $f: M \to \R^3$
if and only if $M$ is diffeomorphic to $S^n$ or to the connected sum of
$S^{n-2}$-bundles over $S^2$ (see Theorem~\ref{classification}).
Now we prove the following theorem.

\begin{thm}\label{main}
Let $f:M \to \R^3$ be a special generic map of a closed simply connected $n$-dimensional manifold,
$n = 5$ or $6$, such that the singular point set $S(f)$ of $f$ has trivial normal bundle in $M$.
Then there exists an embedding $F:M \to W_f  \times \R^{n-2}$ such that $P \circ F = q_f$,
where $P:W_f \times \R^{n-2} \to W_f$ is the projection to the first factor.
\end{thm}

We use the following terminologies in the proof of Theorem~\ref{main}.

\begin{dfn}
Let $X$, $Y$ and $Z$ be smooth manifolds and let $f:X \to Y$, $F:X \to Z$ and $P:Z \to Y$ be smooth maps.
We say that $F$ is a \emph{lift} of $f$ with respect to $P$ if $f=P \circ F$.
In this case, we also say that \emph{$f$} is lifted to $F$ with respect to $P$.
\end{dfn}

\begin{proof}[Proof of Theorem \ref{main}]
We may assume that $M$ is connected.
Then, by Theorem~\ref{stein}, the quotient space $W_f$ has the structure of
a smooth compact simply connected $3$-dimensional manifold with non-empty boundary,
since $M$ is simply connected.
By the solution to the Poincar\'{e} conjecture, this implies that
$W_f$ is diffeomorphic to the $3$-manifold obtained by removing the interior of the union of mutually disjoint
finitely many $3$-balls from the $3$-sphere.

So we have a handle decomposition of $W_f$ as follows:
\[ W_f = (\partial W_f \times [0,1]) \cup \Biggl(\bigcup_{i=1}^{s} h_{i}^1\Biggr) \cup h^3, \]
where $h_i^1$, $i=1,2,\ldots,s$, are $1$-handles and $h^3$ is a $3$-handle.
Let $C$ $(= \partial W_f \times [0,1])$ be the collar neighborhood of $\partial W_f$ in $W_f$,
where $\partial W_f$ corresponds to $\partial W_f \times \{ 0 \}$.

Fix orientations of $M$ and $\R^3$.
Since $\bar{f}:W_f \to \R^3$ is an immersion,
we can orient $W_f$ in such a way that $\bar{f}$ is orientation preserving.
Then, for each $w \in \mathrm{Int}\, W_f$,
we can orient $q_f^{-1}(w) (\cong S^{n-3})$ in such a way that
if $U$ is a small open neighborhood of $w$ in $\mathrm{Int}\, W_f$, and
$\phi:q_{f}^{-1}(U) \to U \times q_{f}^{-1}(w)$ is a local trivialization
with $\phi(x)=(w,x)$ for every $x \in q_{f}^{-1}(w)$, then
$\phi$ is orientation preserving, where the orientations of
$q_{f}^{-1}(U)$ and $U$ are induced from those of $M$ and $W_f$, respectively.

By the assumption that $S(f)$ has trivial normal bundle in $M$,
the composition of the restriction $q_{f}|_{q_{f}^{-1}(C)} : q_{f}^{-1}(C) \to C$ with the natural projection 
\[ p_C:C(= \partial W_f \times [0,1]) \to \partial W_f \]
is a trivial $D^{n-2}$-bundle.
Therefore, we have a bundle trivialization
\[ H_{C}:q_{f}^{-1}(C) \to \partial W_f \times D^{n-2}. \]
We fix an orientation of $\R^{n-2}$, which induces an orientation for $D^{n-2}$.
Then, it induces an orientation for $S^{n-3} = \partial D^{n-2}$.
We may assume that the restriction of $H_C$ to $q_f^{-1}(w)$ is
an orientation preserving diffeomorphism onto
$\{w'\} \times S^{n-3}$ for every $w \in \partial W_f \times \{1 \}$, where
$w'$ is the point in $\partial W_f$ such that $w=(w',1)$.
Then the map 
\[ e_1:q_{f}^{-1}(C) \to C \times \mathbb{R}^{n-2}\]
defined by $e_1(x)=(q_f(x), pr_2 \circ H_C(x))$, $x \in q_f^{-1}(C)$, 
is a smooth map, where the map
\[ pr_2: \partial W_f \times D^{n-2} \to D^{n-2} \subset \R^{n-2}\]
is the projection to the second factor.

Note that $e_1$ is an embedding lift of $q_{f}|_{q_{f}^{-1}(C)}$
with respect to the restriction of $P$ to $C \times \R^{n-2}$.
This is proved as follows.
It is clear that $e_1$ is a lift of $q_{f}|_{q_{f}^{-1}(C)}$
with respect to $P|_{C \times \R^{n-2}}$ by the construction of $e_1$.
Therefore, we have only to prove that $e_1$ is an embedding.
Note that the following diagram commutes:
\[
\xymatrix{
q_f^{-1}(C) \ar[r]^-{q_f} \ar[dr]_-{H_C} & C (= \partial W_f \times [0,1]) \ar[r]^-{p_C} & \partial W_f \\
& \partial W_f \times D^{n-2} \ar[ur]_-{pr_1}, \\
}
\]
where $pr_1$ is the projection to the first factor.
Therefore, the composition
\[ e_1 \circ H_{C}^{-1}:\partial W_f \times D^{n-2} \to C \times \R^{n-2} = \partial W_f \times [0,1] \times \R^{n-2} \]
maps $(x,y)$ to $(x,K(x,y),y)$ for every $x \in \partial W_f$ and $y \in D^{n-2}$, where
$K$ is a smooth map of $\partial W_f \times D^{n-2}$ into $[0,1]$.
This implies that $e_1$ is an embedding.

We will extend $e_1$ to an embedding lift of
the restriction of $q_f$ to $q_f^{-1}(C \cup h_1^1)$ with respect to the restriction
\[ P|_{(C \cup h_1^1) \times \R^{n-2}}:(C \cup h_1^1) \times \R^{n-2} \to C \cup h_1^1.\]
Note that $h_{1}^1$ is identified with $D^2 \times D^1$ and
is attached to $C$ along $D^2 \times S^0$.

Let $\mathrm{Diff}_{+}(S^{n-3})$ be the space of orientation preserving diffeomorphisms of $S^{n-3}$.
By the results of Smale \cite{smale} and Hatcher \cite{Hatcher},
$\mathrm{Diff}_{+}(S^{n-3})$ is homotopy equivalent to $\mathit{SO}(n-2)$, which is connected.

Since the $1$-handle $h_1^1$ is contractible, we have a bundle trivialization
\[ H_{1,1}:q_f^{-1}(h_1^1) \to h_1^1 \times S^{n-3} \]
which induces an orientation preserving diffeomorphism of $q_f^{-1}(w)$ onto
$\{w\} \times S^{n-3}$ for every $w \in h_1^1$.
We have the two end points $(0, \pm 1)$ of the core of the $1$-handle $h_1^1=D^2 \times D^1$.
Then, we define the orientation preserving diffeomorphism
$\phi_{\pm}:S^{n-3} \to S^{n-3}$ as the composition
\[ S^{n-3} = \{(0,\pm 1)\} \times S^{n-3} \to q_f^{-1}(\{(0,\pm 1)\}) \to \{(0,\pm 1)\} \times S^{n-3} = S^{n-3}, \]
where the first map is $H_{1,1}^{-1}$ restricted to $\{(0,\pm 1)\} \times S^{n-3}$,
the second map is $e_1$ restricted to $q_f^{-1}(\{(0,\pm 1)\})$,
and the double-sign corresponds in the same order.

Since $\mathrm{Diff}_{+}(S^{n-3})$ is connected,
there is a continuous path between $\phi_{-}$ and $\phi_{+}$ in $\mathrm{Diff}_{+}(S^{n-3})$.
This induces a homeomorphism
\[ q_f^{-1}(\{0\} \times D^1) (= (\{0\} \times D^1) \times S^{n-3}) \to (\{0\} \times D^1) \times S^{n-3}, \]
which is an orientation preserving diffeomorphism on each $S^{n-3}$-fiber.
This coincides with 
\[ e_1|_{q_f^{-1}(D^2 \times \{-1,+1\})}:q_f^{-1}(D^2 \times \{-1,+1\}) \to (D^2 \times \{-1,+1\}) \times S^{n-3} \]
over $q_f^{-1}(\{0\} \times \{-1,+1\})$.
Therefore, by gluing the two maps, we have a homeomorphism
\[ q_f^{-1}(X) (=X \times S^{n-3}) \to X \times S^{n-3}, \]
where $X=(D^2 \times \{-1,+1\}) \cup (\{0\} \times D^1)$.
By composing this with the natural projection, we have a continuous map
\[ \phi_1:q_f^{-1}(X) (=X \times S^{n-3}) \to S^{n-3}. \]
Let $\phi_2:h_1^1 \times S^{n-3} \to S^{n-3}$ be the continuous map
defined by $\phi_2(x,y)=\phi_1(r(x),y)$ for $x \in h_1^1$ and $y \in S^{n-3}$, where
$r:h_1^1 \to  X $ is a deformation retract.
Then a smooth approximation $\phi_3:h_1^1 \times S^{n-3} \to S^{n-3}$ of $\phi_2$
such that $\phi_3|_{D^2 \times \{-1, +1\} \times S^{n-3}} = \phi_2|_{D^2 \times \{-1, +1\} \times S^{n-3}}$
induces a diffeomorphism of $\{x\} \times S^{n-3}$ to $S^{n-3}$ for every $x \in h_1^1$, since so does $\phi_2$.
Consequently, we have a smooth homeomorphism
\[ \phi_4:h_1^1 \times S^{n-3} \to h_1^1 \times S^{n-3} \]
given by $\phi_4(x,y)=(x,\phi_3(x,y))$, $(x,y) \in h_1^1 \times S^{n-3}$.
Since the derivative of $\phi_4$ at each point is a linear isomorphism,
by the inverse function theorem, we have that $\phi_4$ is a diffeomorphism.
Then the composition
\[ e_2:q_f^{-1}(h_1^1) (=h_1^1 \times S^{n-3}) \to h_1^1 \times S^{n-3} \to h_1^1 \times \R^{n-2}\]
is an embedding lift of $q_f|_{q_f^{-1}(h_1^1)}$ with respect to the restriction of $P$ to $h_1^1 \times \R^{n-2}$, where
the first map is $\phi_4$ and the second map is the product of the identity map of $h_1^1$ and the standard inclusion.
Since $e_1$ and $e_2$ coincide on the intersection of their sources,
by glueing the two maps $e_1$ and $e_2$, we have an embedding lift of
$q_f|_{q_{f}^{-1}(C \cup h_1^1)}$ with respect to $P$ restricted to $(C \cup h_1^1) \times \R^{n-2}$.

By iterating this procedure, we construct an embedding lift $e_3$ of
the restriction of $q_f$ to $q_{f}^{-1}(C \cup (\bigcup_{i=1}^{s} h_i^1))$
with respect to $P$ restricted $(C \cup (\bigcup_{i=1}^{s} h_i^1)) \times \R^{n-2}$.

Now, let us extend the lift $e_3$ to the whole $W_f$.
Since the $3$-handle $h^3$ is contractible, we have a bundle trivialization
\[ H_{3}:q_f^{-1}(h^3) \to h^3 \times S^{n-3}\]
which induces an orientation preserving diffeomorphism of
$q_f^{-1}(w)$ onto $\{w\} \times S^{n-3}$ for every $w \in h^3$.

We define the continuous map $\rho_1:\partial h^3 \times S^{n-3} \to S^{n-3}$ by the composition
\[ \partial h^3 \times S^{n-3} \to q_f^{-1}(\partial h^3) \to \partial h^3 \times S^{n-3} \to S^{n-3}, \]
where the first map is the restriction of $H_{3}^{-1}$ to $\partial h^3 \times S^{n-3}$,
the second map is the restriction of $e_3$ to $q_f^{-1}(\partial h^3)$,
and the last map is the projection to the second factor.
Note that $\rho_1$ induces an orientation preserving diffeomorphism of
$\{w\} \times S^{n-3}$ onto $S^{n-3}$ for every $w \in \partial  h^3$.

Recall that $\mathrm{Diff}_{+}(S^{n-3})$ is homotopy equivalent to $\mathit{SO}(n-2)$ (see \cite{Hatcher,smale});
hence, it is $2$-connected.
Therefore, the continuous map $\rho_1$ extends to
a continuous map $\rho_2:h^3 \times S^{n-3} \to S^{n-3}$ which induces
an orientation preserving diffeomorphism of $\{w\} \times S^{n-3}$ onto $S^{n-3}$ for every $w \in h^3$.
Then a smooth approximation $\rho_3:h^3 \times S^{n-3} \to S^{n-3}$ of $\rho_2$
such that $\rho_3|_{\partial h^3 \times S^{n-3}} = \rho_2|_{\partial h^3 \times S^{n-3}}$ induces
a diffeomorphism of $\{w\} \times S^{n-3}$ onto $S^{n-3}$ for every $w \in h^3$. 
So we have a smooth homeomorphism
\[ \rho_4:h^3 \times S^{n-3} \to h^3 \times S^{n-3} \]
given by $\rho_4(x,y)=(x,\rho_3(x,y))$ for $x \in h^3$ and $y \in S^{n-3}$.

Since the derivative of $\rho_4$ at each point is a linear isomorphism,
by the inverse function theorem, we have that $\rho_4$ is a diffeomorphism.
Then the composition
\[ e_4:q_f^{-1}(h^3) (=h^3 \times S^{n-3}) \to h^3 \times S^{n-3} \to h^3 \times \R^{n-2}\]
is an embedding lift of $q_f|_{q_f^{-1}(h^3)}$ with respect to the restriction of $P$ to $h^3 \times \R^{n-2}$, where
the first map is $\rho_4$ and the second map is the product of the identity map of $h^3$ and the standard inclusion.
Since $e_3$ and $e_4$ coincide on the intersection of their sources,
by gluing the two maps $e_3$ and $e_4$, we have an embedding lift of $q_f$ with respect to $P$.
This completes the proof.
\end{proof}

\begin{rem}\label{Crowley-isomorphism}
Note that the key ingredient in the proof of Theorem~\ref{main} is that
$\pi_2 \mathrm{Diff}_{+}(S^{n-3}) = 0$ for $n=5, 6$.
On the other hand, Crowley--Schick \cite{crowley} proved that for every $j \geq 1$, we have
\[ \pi_2 \mathrm{Diff}(D^{8j-1}, \partial) \neq 0, \]
where $\mathrm{Diff}(D^{8j-1}, \partial)$ is the space of diffeomorphisms of $D^{8j-1}$
which are the identity on some neighborhood of $\partial D^{8j-1}$.
It is known (see Proposition~4 of Appendix in \cite{Cerf}, for example) that the following homotopy equivalence holds: 
\[ \mathrm{Diff}_{+}(S^n) \simeq \mathrm{Diff}(D^n, \partial) \times \mathit{SO}(n+1). \]
Therefore, the result mentioned above implies that
\[ \pi_2(\mathrm{Diff}_{+}(S^{8j-1})) \neq 0, \]
for every $j \geq 1$.
This means that the method in the proof of Theorem~\ref{main}
does not work in higher dimensions in general.
\end{rem}

As a consequence of Theorem~\ref{main}, we have the following result.

\begin{thm}\label{main-cor}
Let $f:M \to \R^3$ be a special generic map of
a closed simply connected $n$-dimensional manifold with $n=5,6$ and
let $\pi:\R^{n+1} \to \R^3$ be the standard projection.
Then the following conditions are all equivalent to each other:
\begin{enumerate}
\item There exists an embedding $F_1:M \to \R^{n+1}$ such that $\pi \circ F_1 = f$.
\item There exists an immersion $F_2:M \to \R^{n+1}$ such that $\pi \circ F_2 = f$.
\item The singular point set $S(f)$ of $f$ has a trivial normal bundle in $M$.
\item The manifold $M$ is spin.
\end{enumerate}
\end{thm}

\begin{proof}
Assume that there exists an immersion $F:M \to \R^{n+1}$ such that $\pi \circ F = f$.
Then the map $F':M \to \R^{n+1} \times \R$ defined by $F'(x)=(F(x),0)$, $x \in M$, is also an immersion lift of $f$.
Therefore, by Proposition \ref{lemmaA}, we conclude that the singular point set $S(f)$ of $f$ has a trivial normal bundle in $M$.

Now, suppose that $S(f)$ has a trivial normal bundle in $M$.
Then by Theorem~\ref{main}, we get an embedding lift $e:M \to W_f \times \R^{n-2}$ of $q_f$.
Since $W_f$ can be embedded into $\R^{n-2}$, by using
an embedding $\bar{e}:W_f \to \R^{n-2}$ and the immersion $\bar{f}:W_f \to \R^3$,
we get an embedding $G=(\bar{f},\bar{e}):W_f \to \R^{n+1}$ such that
$G$ is an embedding lift of $\bar{f}$ with respect to the natural projection $\R^{n+1} \to \R^3$.
Since $W_f$ is compact, there exists an embedding $H:W_f \times \R^{n-2} \to \R^{n+1}$
such that $H(x,0)=G(x)$ for every $x \in W_f$ and the diagram
\[
\xymatrix{
W_f \times \R^{n-2} \ar[d]^-P \ar[r]^-H & \R^{n+1} \ar[d]^-{\pi}\\
W_f \ar[r]_-{\bar{f}} & \R^3
}
\]
commutes.
Then the composition of $e$ with $H$ is an embedding lift of $f$.

It is trivial that the first condition implies the second one.
Thus, the first three conditions in Theorem~\ref{main-cor} are equivalent to each other.

Saeki--Takase proved that the second and the last conditions are equivalent to each other (see Theorem~6.1 in \cite{ST}).
This completes the proof of Theorem~\ref{main-cor}.
\end{proof}

\begin{rem}
Note that we can directly prove that items 3 and 4 in Theorem \ref{main-cor}
are equivalent to each other without using a result of Saeki--Takase \cite{ST} as follows.
Recall that $M$ is diffeomorphic to $\partial E$ for some linear $D^{n-2}$-bundle $E$ over $W_f$ by Theorem~\ref{fiber}.
Since $W_f$ is simply connected, we have
\[ W_f \cong W_1 \natural W_2 \natural \cdots \natural W_{b}, \]
where the symbol ``$\natural$'' denotes boundary connected sum,
$W_i \cong S^2 \times [0,1]$ $(i=1,2,\ldots,b)$ and
$b \geq 0$ (when $b=0$, $W_f \cong D^3$).
Let $E_i$ $(i=1,2,\ldots,b)$ be the $D^{n-2}$-bundle over $W_i$ induced from the inclusion $W_i \hookrightarrow W_f$. 
Then, we have
\[ M \cong \partial E_1 \, \sharp \, \partial E_2 \, \sharp \cdots \sharp \, \partial E_b.\]
Note that the manifold $\partial E_i$ is the total space of an $S^{n-2}$-bundle over $S^2$.
This is a spin manifold if and only if the bundle $E_i$ is trivial.
Therefore, the manifold $M$ is spin if and only if all the bundles $E_i$ are trivial.
Finally, it is easy to see that this last condition is equivalent to the triviality of the normal bundle $\nu_f$ of $S(f)$ in $M$.
\end{rem}

The following proposition shows that for $n \geq 4$,
there exist  special generic maps of closed simply connected $n$-dimensional manifolds
into $\R^3$ with trivial normal bundle of the singular point set.
By virtue of Theorem~\ref{main-cor}, such special generic maps for $n = 5, 6$
can be lifted to embeddings in codimension one.

\begin{prop}\label{ex1}
For $n \geq 3$, there is a special generic map $f:S^{n-2} \times S^2 \to \R^3$
such that the normal bundle $\nu_f$ of $S(f)$ in $S^{n-2} \times S^2$ is trivial.
\end{prop}

\begin{proof}
Let $h:S^{n-2} \to \R$ be the Morse function given by
\[ h(x_1,x_2,\ldots,x_{n-1}) = x_{n-1} \]
for $(x_1,x_2,\ldots,x_{n-1}) \in S^{n-2} \subset \R^{n-1}$.
Then the composition
\[ S^{n-2} \times S^2 \xrightarrow{h \times \mathrm{id}} \R \times S^2 \to \R^3 \]
is a special generic map, where $\mathrm{id}$ is the identity map of $S^2$, and
the last map is the composition of a trivialization of the open tubular neighborhood of $S^2$ in $\R^3$ with the inclusion map.
Note that $S(f)= \{ (0,0,\ldots,0,\pm 1) \} \times S^2$ is the disjoint union
of two $2$-spheres and it has trivial normal bundle in $S^{n-2} \times S^2$.
\end{proof}

On the other hand, the following proposition implies that
there exist  special generic maps of closed simply connected $n$-dimensional manifolds
into $\R^3$ with non-trivial normal bundle of the singular point set.
By virtue of Theorem~\ref{main-cor}, such special generic maps for $n = 5, 6$ cannot be lifted to embeddings in codimension one.

\begin{prop}\label{ex2}
For $n \geq 4$, there is a special generic map $f:M \to \R^3$
such that the normal bundle $\nu_f$ of $S(f)$ in $M$ is non-trivial, where
$M$ is a non-trivial $S^{n-2}$-bundle over $S^2$.
\end{prop}

\begin{proof}
For real numbers $t$ with $0 \leq t \leq 2\pi$,
we define the diffeomorphism $g_{t}:S^{n-2} \to S^{n-2}$ by 
\[ g_{t}(x_1,x_2,\ldots,x_{n-1})=
(x_1\cos t - x_2 \sin t,x_1\sin t+x_2\cos t,x_3,x_4,\ldots,x_{n-1}) \]
for $(x_1,x_2,\ldots,x_{n-1}) \in S^{n-2}$.
Note that we have $g_0=g_{2\pi}$.
By using this map, we define the diffeomorphism
$\Phi : S^{n-2} \times \partial D^2 \to S^{n-2} \times \partial D^2$ by
\[\Phi(x,(\cos t, \sin t)) = (g_t(x),(\cos t, \sin t)) \]
for $x \in S^{n-2}$ and $0 \leq t \leq 2\pi$.
Pasting $S^{n-2} \times D^2$ and its copy along the boundary by $\Phi$,
we obtain the closed $n$-dimensional manifold $M$.
It is easy to see that $M$ is a non-trivial $S^{n-2}$-bundle over $S^2$.

Now, we define the special generic map $h:S^{n-2} \to \mathbb{R}$ by
\[ h(x_1,x_2,\ldots,x_{n-1})=x_{n-1} \]
for $(x_1,x_2,\ldots,x_{n-1}) \in S^{n-2}$.
Then we have
\[ (h \times \mathrm{id}) \circ \Phi = h \times \mathrm{id} :
S^{n-2} \times \partial D^2 \to \mathbb{R} \times  \partial D^2. \]
Therefore, the map
\[ (h \times \mathrm{id}) \cup (h \times \mathrm{id}) :
M = (S^{n-2} \times D^2) \cup_{\Phi} (S^{n-2} \times D^2)
\rightarrow (\mathbb{R} \times D^2) \cup (\mathbb{R} \times D^2) \]
is well-defined, where
\[ (\mathbb{R} \times D^2) \cup (\mathbb{R} \times D^2) = \R \times S^2 \]
is the space obtained by pasting
$\mathbb{R} \times D^2$ and its copy along the boundary by the identity map.
So we obtain the composition map
\[ f : M \longrightarrow \mathbb{R} \times S^2 \rightarrow \mathbb{R}^3, \]
where the last map is the composition of a trivialization of the open tubular neighborhood of
$S^2$ in $\mathbb{R}^3$ with the inclusion map.
Since the second Stiefel-Whieney class of $\nu_f$ does not vanish, 
we see that the map $f:M \to \mathbb{R}^3$ is a special generic map and that
$S(f)$ is the disjoint union of two $2$-spheres with non-trivial normal bundle in $M$
\end{proof}

The map $f$ in the following proposition cannot be lifted to an embedding
into $\R^{n+1}$ by a result in \cite{ST}.
So Theorems~\ref{main} and \ref{main-cor} do not hold if we drop the condition
that $M$ should be simply  connected.

\begin{prop}
For $n=5, 6$, there exists a special generic map $f:M \to \R^3$ of
a closed orientable $n$-dimensional manifold $M$ such that
the normal bundle $\nu_f$ of $S(f)$ in $M$ is trivial
and that $M$ is neither spin nor simply connected.
\end{prop}

\begin{proof}
Let $\pi_1:E_1 \to S^2$ be the projection of the non-trivial orientable linear $D^{n-2}$-bundle over $S^2$
(such a bundle uniquely exists up to isomorphism since $\pi_1(\mathit{SO}(n-2))=\mathbb{Z}_2$).
Then the product map $\pi_2=\pi_1 \times \mathrm{id}_{S^1}:E_1 \times S^1 \to S^2 \times S^1$
is a non-trivial orientable linear $D^{n-2}$-bundle over $S^2 \times S^1$.
Set $W=S^2 \times S^1 \setminus \mathrm{Int}\, D_1$,  $E=\pi_2^{-1}(W)$ and
$\pi=\pi_2|_{E}:E \to W$, where $D_1$ is a $3$-ball in $S^2 \times S^1$.
Then, $\pi:E \to W$ is a non-trivial orientable linear $D^{n-2}$-bundle over $W$.
It is clear that $W$ can  be immersed in $\R^3$.
Then, by using the same method used in the proof of \cite[Proposition~2.1]{saeki},
we can construct a special generic map $f:M \to \R^3$ such that $M=\partial E$, $W_f=W$ and
$q_f:M \to W_f$ coincide with $\pi$ over $q_f^{-1}(C)$ for some collar neighborhood $C$ of $\partial W_f$ in $W_f$.
Note that the normal bundle $\nu_f$ of $S(f)$ in $M$ is trivial, since $\pi:E \to W$ is trivial over $\partial W$.
This completes the proof.
\end{proof}

The following theorem is proved by a method similar to that used in the proofs of 
Theorems~\ref{main} and \ref{main-cor}.

\begin{thm}\label{main2}
Let $f:M \to \R^3$ be a special generic map of
a closed orientable $n$-dimensional manifold, $n \geq 5$.
Then the quotient map $q_f:M \to W_f$ lifts to an embedding into $W_f \times \R^{k-3}$
with respect to the projection $P:W_f \times \R^{k-3} \to W_f$
if the normal bundle $\nu_f$ of the singular point set $S(f)$ in $M$
is  trivial and $k \geq (3n+3)/2$.
\end{thm}

We need the following proposition to prove Theorem~\ref{main2}.

\begin{prop}[Budney, \cite{budney}]\label{budney-lemma}
The embedding space $\mathrm{Emb}(S^n,\R^k)$ is
$\min \{ 2k-3n-4, k-n-2 \}$-connected
if $k \geq n+2 \geq 3$.
\end{prop}

\begin{proof}[Proof of Theorem \ref{main2}]
By Proposition~\ref{budney-lemma}, since $k \geq (3n+3)/2$ and $n \geq 5$,
we have that the embedding space $\mathrm{Emb}(S^{n-3},\R^{k-3})$ is $2$-connected.
This is a key to proving Theorem~\ref{main2}.

We may assume that $M$ is connected.
Then, by Theorem~\ref{stein}, the quotient space $W_f$
has the structure of a smooth compact orientable connected
$3$-dimensional manifold with non-empty boundary.
So we have a handle decomposition of $W_f$ as follows:
\[ W_f = (\partial W_f \times [0,1]) \cup \Biggl(\bigcup_{i=1}^{s} h_{i}^1\Biggr)
\cup \Biggl(\bigcup_{j=1}^{t} h_{j}^2\Biggr) \cup h^3, \]
where $h_i^1$, $i=1,2,\ldots,s$, are $1$-handles,
$h_j^2$, $j=1,2,\ldots,t$, are $2$-handles, and $h^3$ is a $3$-handle.
Let $C (= \partial W_f \times [0,1])$ be the collar neighborhood of $\partial W_f$ in $W_f$.
Here, $\partial W_f$ corresponds to $\partial W_f \times \{ 0 \}$.

By using the same method used in the proof of Theorem~\ref{main}, we can construct
an embedding lift $e_1$ of $q_{f}|_{q_{f}^{-1}(C)}$ with respect to $P|_{C \times \R^{k-3}}$.

We will extend $e_1$ to an embedding lift of
the restriction of $q_f$ to $q_f^{-1}(C \cup h_1^1)$ with respect to the restriction
\[ P|_{(C \cup h_1^1) \times \R^{k-3}}:(C \cup h_1^1) \times \R^{k-3} \to C \cup h_1^1.\]
Note that $h_{1}^1$ is identified with $D^2 \times D^1$ and
is attached to $C$ along $D^2 \times S^0$.

Since the $1$-handle $h_1^1$ is contractible, we have a bundle trivialization
\[ H_{1,1}:q_f^{-1}(h_1^1) \to h_1^1 \times S^{n-3}. \]
We have two end points $(0,\pm 1)$ of the core of the $1$-handle $h_1^1=D^2 \times D^1$.
Then, we define the embeddings
$\phi_{\pm}:S^{n-3} \to \R^{k-3}$ as the composition
\[ S^{n-3} = \{(0,\pm 1)\} \times S^{n-3} \to q_f^{-1}(\{(0,\pm 1)\}) \to \{(0,\pm 1)\} \times \R^{k-3} = \R^{k-3}, \]
where the first map is $H_{1,1}^{-1}$ restricted to $\{(0,\pm 1)\} \times S^{n-3}$,
the second map is $e_1$ restricted to $q_f^{-1}(\{(0,\pm 1)\})$
and the double-sign corresponds in the same order.

Since $\mathrm{Emb}(S^{n-3},\R^{k-3})$ is connected,
there is a continuous path between $\phi_{-}$ and $\phi_{+}$ in $\mathrm{Emb}(S^{n-3},\R^{k-3})$.
This induces a topological embedding
\[ q_f^{-1}(\{0\} \times D^1) (= (\{0\} \times D^1) \times S^{n-3}) \to (\{0\} \times D^1) \times \R^{k-3}, \]
which is an embedding on each $S^{n-3}$-fiber. This coincides with 
\[ e_1|_{q_f^{-1}(D^2 \times \{-1,+1\})}:q_f^{-1}(D^2 \times \{-1,+1\}) \to (D^2 \times \{-1,+1\}) \times \R^{k-3} \]
over $q_f^{-1}(\{0\} \times \{-1,+1\})$.
Therefore, by gluing the two maps, we have a topological embedding
\[ q_f^{-1}(X) (=X \times S^{n-3}) \to X \times \R^{k-3}, \]
where $X=(D^2 \times \{-1,+1\}) \cup (\{0\} \times D^1)$.
By composing this with the natural projection, we have a continuous map
\[ \phi_1:q_f^{-1}(X) (=X \times S^{n-3}) \to \R^{k-3}. \]
Let $\phi_2:h_1^1 \times S^{n-3} \to \R^{k-3}$ be the continuous map
defined by $\phi_2(x,y)=\phi_1(r_1(x),y)$ for $x \in h_1^1$ and $y \in S^{n-3}$, where
$r_1:h_1^1 \to  X $ is a deformation retract.
Then a smooth approximation $\phi_3:h_1^1 \times S^{n-3} \to \R^{k-3}$ of $\phi_2$
such that $\phi_3|_{D^2 \times \{-1,+1\} \times S^{n-3}}=\phi_2|_{D^2 \times \{-1,+1\} \times S^{n-3}}$
induces a smooth embedding of $\{x\} \times S^{n-3}$ into $\R^{k-3}$ for every $x \in h_1^1$, since so does $\phi_2$.
Consequently, we have a smooth injection
\[ \phi_4:h_1^1 \times S^{n-3} \to h_1^1 \times \R^{k-3} \]
defined by $\phi_4(x,y)=(x,\phi_3(x,y))$, $(x,y) \in h_1^1 \times S^{n-3}$.
Since the derivative of $\phi_4$ at each point is injective and $h_1^1 \times S^{n-3}$ is compact,
we have that $\phi_4$ is an embedding.
Put $e_2=\phi_4$. Then $e_2$ is an embedding lift of $q_f|_{q_f^{-1}(h_1^1)}$
with respect to $P$ restricted to $h_1^1 \times \R^{k-3}$.
Since $e_1$ and $e_2$ coincide on the intersection of their sources,
by glueing the two maps $e_1$ and $e_2$, we have an embedding lift of
$q_f|_{q_{f}^{-1}(C \cup h_1^1)}$ with respect to $P$ restricted to $(C \cup h_1^1) \times \R^{k-3}$.

By iterating this procedure, we construct an embedding lift $e_3$ of
$q_f$ restricted to $q_{f}^{-1}(C \cup (\bigcup_{i=1}^{s} h_i^1))$
with respect to $P$ restricted to $(C \cup (\bigcup_{i=1}^{s} h_i^1)) \times \R^{k-3}$.

Put $W_1=C \cup (\bigcup_{i=1}^{s} h_i^1)$.
We will extend $e_3$ to an embedding lift of
the restriction of $q_f$ to $q_f^{-1}(W_1 \cup h_1^2)$ with respect to the restriction
\[ P|_{(W_1 \cup h_1^2) \times \R^{k-3}}:(W_1 \cup h_1^2) \times \R^{k-3} \to W_1 \cup h_1^2.\]
Note that $h_{1}^2$ is identified with $D^1 \times D^2$ and
is attached to $W_1$ along $D^1 \times S^1$.

Since the $2$-handle $h_1^2$ is contractible, we have a bundle trivialization
\[ H_{1,2}:q_f^{-1}(h_1^2) \to h_1^2 \times S^{n-3}. \]
We have the circle $\{0\} \times S^1$ as the core of the attaching annulus of the $2$-handle $h_1^2=D^1 \times D^2$.
Then, we define the embedding $\psi_{t}:S^{n-3} \to \R^{k-3}$ $(t \in S^1)$ as the composition
\[ S^{n-3} = \{(0,t)\} \times S^{n-3}  \to q_f^{-1}(\{(0,t)\}) \to
\{(0,t)\} \times \R^{k-3} = \R^{k-3}, \]
where the first map is $H_{1,2}^{-1}$ restricted to $\{(0,t)\} \times S^{n-3}$ and
the second map is $e_3$ restricted to $q_f^{-1}(\{(0,t)\})$.
The family $\{\psi_t\}_{t \in S^1}$ induces a continuous map $\psi_0$ of  $S^1$ into $\mathrm{Emb}(S^{n-3},\R^{k-3})$.

Since $\pi_1\mathrm{Emb}(S^{n-3},\R^{k-3})=0$,
the map $\psi_0$ extends to a continuous map of $D^2$ into $\mathrm{Emb}(S^{n-3},\R^{k-3})$.
This induces a topological embedding
\[ q_f^{-1}(\{0\} \times D^2) (= (\{0\} \times D^2) \times S^{n-3}) \to (\{0\} \times D^2) \times \R^{k-3}, \]
which is an embedding on each $S^{n-3}$-fiber. This coincides with 
\[ e_3|_{q_f^{-1}(D^1 \times S^1)}:q_f^{-1}(D^1 \times S^1) \to (D^1 \times S^1) \times \R^{k-3} \]
over $q_f^{-1}(\{0\} \times S^1)$.
Therefore, by gluing the two maps, we have a topological embedding
\[ q_f^{-1}(Y) (=Y \times S^{n-3}) \to Y \times \R^{k-3}, \]
where $Y=(D^1 \times S^1) \cup (\{0\} \times D^2)$.
By composing this with the natural projection, we have a continuous map
\[ \psi_1:q_f^{-1}(Y) (=Y \times S^{n-3}) \to \R^{k-3}. \]
Let $\psi_2:h_1^2 \times S^{n-3} \to \R^{k-3}$ be the continuous map
defined by $\psi_2(x,y)=\psi_1(r_2(x),y)$ for $x \in h_1^2$ and $y \in S^{n-3}$, where
$r_2:h_1^2 \to  Y$ is a deformation retract.
Then a smooth approximation $\psi_3:h_1^2 \times S^{n-3} \to \R^{k-3}$ of $\psi_2$
such that $\psi_3|_{D^1 \times S^1 \times S^{n-3}}=\psi_2|_{D^1 \times S^1 \times S^{n-3}}$
induces a smooth embedding of $\{x\} \times S^{n-3}$ into $\R^{k-3}$ for every $x \in h_1^2$, since so does $\psi_2$.
Consequently, we have a smooth injection
\[ \psi_4:h_1^2 \times S^{n-3} \to h_1^2 \times \R^{k-3} \]
defined by $\psi_4(x,y)=(x,\psi_3(x,y))$, $(x,y) \in h_1^2 \times S^{n-3}$.
Since the derivative of $\psi_4$ at each point is injective and $h_1^2 \times S^{n-3}$ is compact,
we have that $\psi_4$ is an embedding.
Put $e_4=\psi_4$. Then $e_4$ is an embedding lift of $q_f|_{q_f^{-1}(h_1^2)}$
with respect to $P$ restricted to $h_1^2 \times \R^{k-3}$.
Since $e_3$ and $e_4$ coincide on the intersection of their sources,
by glueing the two maps $e_3$ and $e_4$, we have an embedding lift of
$q_f|_{q_{f}^{-1}(W_1 \cup h_1^2)}$ with respect to $P$ restricted to $(W_1 \cup h_1^2) \times \R^{k-3}$.

By iterating this procedure, we construct an embedding lift $e_5$ of
$q_f$ restricted to $q_{f}^{-1}(W_1 \cup (\bigcup_{j=1}^{t} h_j^2))$
with respect to $P$ restricted to $(W_1 \cup (\bigcup_{j=1}^{t} h_j^2)) \times \R^{k-3}$.

Now, let us extend the lift $e_5$ to the whole $W_f$.
Since the $3$-handle $h^3$ is contractible, we have a bundle trivialization
\[ H_{3}:q_f^{-1}(h^3) \to h^3 \times S^{n-3}.\]
Then, we define the continuous map $\rho_1:\partial h^3 \times S^{n-3} \to \R^{k-3}$ by the composition
\[ \partial h^3 \times S^{n-3} \to q_f^{-1}(\partial h^3) \to \partial h^3 \times \R^{k-3} \to \R^{k-3}, \]
where the first map is the restriction of $H_{3}^{-1}$ to $\partial h^3 \times S^{n-3}$,
the second map is the restriction of $e_5$ to $q_f^{-1}(\partial h^3)$
and the last map is the projection.

Recall that $\pi_2 \mathrm{Emb}(S^{n-3},\R^{k-3})=0$.
Therefore, the continuous map $\rho_1$ extends to
a continuous map $\rho_2:h^3 \times S^{n-3} \to \R^{k-3}$ which induces
a smooth embedding into $\R^{k-3}$ on each $S^{n-3}$-fiber.
Then a smooth approximation $\rho_3:h^3 \times S^{n-3} \to \R^{k-3}$ of $\rho_2$
such that $\rho_3|_{\partial h^3 \times S^{n-3}} = \rho_2|_{\partial h^3 \times S^{n-3}}$ induces
a smooth embedding on each $S^{n-3}$-fiber, since so does $\rho_2$.
So we have a smooth injection
\[ \rho_4:h^3 \times S^{n-3} \to h^3 \times \R^{k-3} \]
defined by $\rho_4(x,y)=(x,\rho_3(x,y))$ for $x \in h^3$ and $y \in S^{n-3}$.

Since the derivative of $\rho_4$ at each point is injective and $h^3 \times S^{n-3}$ is compact,
we have that $\rho_4$ is a smooth embedding.
Put $e_6=\rho_4$.
Then, since $e_5$ and $e_6$ coincide on the intersection of their sources,
by gluing the two maps $e_5$ and $e_6$, we have an embedding lift of $q_f$ with respect to $P$.
This completes the proof.
\end{proof}

As a consequence of Theorem~\ref{main2},
by using the same method as that in the proof of Theorem~\ref{main-cor},
we have the following result.

\begin{thm}\label{main2-cor}
Let $f:M \to \R^3$ be a special generic map of
a closed orientable $n$-dimensional manifold, $n \geq 5$.
Then $f$ lifts to an embedding into $\R^k$ with respect to the natural projection $\pi:\R^k \to \R^3$
if the normal bundle $\nu_f$ of the singular point set $S(f)$ in $M$ is  trivial and $k \geq (3n+3)/2$.
\end{thm}

Furthermore, Theorem~\ref{main2-cor} can be generalized as follows.

\begin{thm}\label{generalization}
Let $f:M \to \R^p$ be a special generic map of
a closed orientable $n$-dimensional manifold with $n>p \geq 1$.
Then $f$ lifts to an embedding into $\R^k$ with respect to the natural projection $\pi:\R^k \to \R^p$
if the normal bundle $\nu_f$ of the singular point set $S(f)$ in $M$ is  trivial and $k \geq \max\{(3n+3)/2,n+p+1\}$.
\end{thm}

This is proved by the (almost) same method as that in the proofs of Theorems~\ref{main2} and \ref{main2-cor}.
The key ingredient to proving Theorem~\ref{generalization} is that
$\mathrm{Emb}(S^{n-p},\R^{k-p})$ is $(p-1)$-connected.

\section{Acknowledgements}
The author would like to express his sincere gratitude to
Professor Osamu Saeki for useful comments and suggestions.
The author would also like to express his thanks to
Professor Masamichi Takase for invaluable comments and suggestions.

\end{document}